\documentclass[11pt]{amsart}
\usepackage{graphicx, subfigure}
\usepackage{txfonts}

\newtheorem{theorem}{Theorem}[section]
\newtheorem{proposition}[theorem]{Proposition}
\newtheorem{lemma}[theorem]{Lemma}

\theoremstyle{remark}

\begin{document}
\baselineskip 16pt

\thispagestyle{empty}

\begin{center}\sf
{\Large Slow motion for compressible}\vskip.15cm
{\Large isentropic Navier--Stokes equations}\vskip.5cm

{\tt \today}\vskip.25cm

Corrado MASCIA\footnote{Dipartimento di Matematica ``G. Castelnuovo'', 
	Sapienza, Universit\`a di Roma, P.le Aldo Moro, 2 -- 00185 Roma (ITALY),
	\&
	Istituto per le Applicazioni del Calcolo, 
	Consiglio Nazionale delle Ricerche (ITALY),
	associated in the framework of the program ``Intracellular Signalling'',
  	\texttt{\tiny mascia@mat.uniroma1.it}},
Marta STRANI\footnote{Dipartimento di Matematica ``G. Castelnuovo'',
  	Sapienza -- Universit\`a di Roma, P.le Aldo Moro, 2 - 00185 Roma (ITALY),
  	\texttt{\tiny strani@mat.uniroma1.it}}
\end{center}
\vskip.5cm

\begin{quote}\small \baselineskip 14pt 
{\sf Abstract.}
We consider the compressible Navier--Stokes equations for isentropic dynamics
with real viscosity on a bounded interval.
In the case of boundary data defining an admissible shock wave for the corresponding
unviscous hyperbolic system, we determine a scalar differential equation describing 
the motion of the internal transition layer.
In particular, for viscosity $\varepsilon$ small, the velocity of the motion is exponentially small.
The approach is based on the construction of a 1-parameter manifold of approximate
solutions and on an appropriate projection of the evolution of the complete Navier--Stokes
system towards such manifold.
\vskip.15cm

\vskip.15cm

\end{quote}

\baselineskip=16pt

\section{Introduction}\label{sec:introduction}

This article is devoted to the description of slow-motion for the hyperbolic-parabolic
Navier--Stokes system for compressible isentropic fluid with real viscosity, that is,
in term of the variables density/velocity $(\rho,w)$
\begin{equation}\label{isentro}
	\frac{\partial \rho}{\partial t}+\frac{\partial(\rho\,w)}{\partial x}=0,\qquad
	\frac{\partial(\rho\,w)}{\partial t}+\frac{\partial}{\partial x}\left\{\rho\,w^2+ P(\rho)
			-\varepsilon\nu(\rho)\frac{\partial w}{\partial x}\right\}=0.
\end{equation}
The functions definining the pressure $P$ and the viscosity $\nu$ are required to
satisfy the standard assumptions
\begin{equation*}
	P'(\rho)>0,\qquad P''(\rho)>0,\qquad \nu(\rho)>0
\end{equation*}
for any $\rho$ under consideration.
The relevant cases $P(\rho)=C\,\rho^\alpha$ with $\alpha>1$ and $\nu(\rho)=C\,\rho^\beta$
with $\alpha,\beta\geq 1$ and $C>0$ fit into the general framework
In particular, shallow water Saint--Venant system with viscosity corresponds to the
case $\alpha=2$, $\beta=1$ (see \cite{GerbPert01} for the derivation of the model
and \cite{BresDesjMeti07} for a recent review on shallow water equations).

Given $\ell>0$, the space variable $x$ belongs to the bounded interval $(-\ell,\ell)$
and system \eqref{isentro} is complemented with boundary conditions
\begin{equation*}
	\rho(-\ell)=\rho_-\,\quad w(\pm\ell)=w_\pm>0.
\end{equation*}
Let us set
\begin{equation*}
	F(u,v)=\frac{v^2}{u}+ P(u).
\end{equation*}
Considering the variables density/momentum $(u,v)=(\rho,\rho\,w)$,
system \eqref{isentro} becomes
\begin{equation}\label{equv}
	\frac{\partial u}{\partial t}+\frac{\partial v}{\partial x}=0,\qquad
	\frac{\partial v}{\partial t}+\frac{\partial}{\partial x}
		\left\{F(u,v)-\varepsilon\nu(u)\frac{\partial}{\partial x}\left(\frac{v}{u}\right)\right\}=0,
\end{equation}
with the boundary conditions
\begin{equation}\label{bdaryuv}
	w_\pm\,u(\pm\ell)-v(\pm\ell)=0,\qquad	v(-\ell)=\rho_-w_-
\end{equation}
In the limiting case $\ell\to +\infty$, system \eqref{equv} is known to support 
traveling wave solutions, i.e. solution with the form $(u,v)=(U(x-ct),V(x-ct))$ 
satisfying the asymptotic conditions
\begin{equation*}
	(U,V)(-\infty)=(u_-,v_-),\qquad (U,V)(+\infty)=(u_+,v_+)
\end{equation*}
for appropriate choices of $(u_\pm,v_\pm)$ and $c$ given by the usual
Rankine--Hugoniot relation.
Thanks to the galileian invariance of \eqref{equv}, we may assume, without loss
of generality, the speed $c$ to be zero.
In such a case, the component $V$ turns to be constant, so that $v_+$ and $v_-$ are
forced to be equal to a common value, denoted here by $v_\ast$.
Hence, the equivalent form of \eqref{bdaryuv} for the special solution $(U,V)$ is
\begin{equation*}
	U(\pm\ell)=u_\pm:=\frac{v_\ast}{w_\pm},\qquad		V(-\ell)=v_\ast
\end{equation*}
The wave solution $(U,V)$ belongs to a 1-parameter family of stationary solutions
to \eqref{equv} generated by the space translation group.
The stability analysis of such family has been explored for a long time and 
a number of orbital stability results for different regimes and structure functions has
been proved (see \cite{MatsNish85, BHRZ08, HumpLafiZumb10, MatsWang10}).

The analysis of the dynamics in bounded domains has been also investigated
(among others, we quote \cite{BresDesjGera07, LiLiXin08, LianGuoLi10}), but 
always with a limited attention to the dynamical behavior close to equilibrium
configurations.
In term of special solution, in bounded intervals $(-\ell,\ell)$ with fixed $\ell>0$,
the situation is different with respect to case of the whole real line.
Formally setting $\varepsilon=0$ in \eqref{equv}, there still exists a one-parameter
family of steady states given by a sharp transition at some point $\xi\in(-\ell,\ell)$.
Differently,  for $\varepsilon>0$, there exists a single steady state satisfying the boundary
conditions \eqref{bdaryuv} with $w_\pm, v_-$ (or, equivalently,
$u_\pm, v_\ast$) chosen so that there exists a stationary wave in the whole real line
with same asymptotic states.

As a consequence, in analogy to the case of scalar viscous conservation laws, it is expected
that, in the regime of $\varepsilon$ small, the solution determined by an initial datum consisting
of a single transition from $(u_-,v_\ast)$ to $(u_+,v_\ast)$ converges in a short timescale to a
specific profile with transition located at some point $\xi=\xi(t)$ and, then, on a much longer
timescale, moves to the location of the single steady state.
Such motion is an effect of the boundary data and it is expected to be very slow.
Our aim is to derive formally a differential equation for the location $\xi$ and to show that 
the motion of $\xi$ is indeed exponentially slow for $\varepsilon\approx 0$.
The first tool is the definition of an {\it approximate invariant manifold}
$\{W(\cdot,\xi)\,:\,\xi\in(-\ell,\ell)\}$ whose elements are approximate steady state 
of \eqref{equv} and resemble, in a sense, transitions from $(u_-,v_\ast)$ to $(u_+,v_\ast)$
located at $\xi$.
Among the many different and significant choices for constructing such manifold,
our preference goes to the one we explored in \cite{MascStra13} for the case
of scalar conservation laws and that consists in matching exact steady states
in $(-\ell,\xi)$ and in $(\xi,\ell)$ at $x=\xi$ by imposing some appropriate conditions.
Assuming that the spectrum of the linear operator $\mathcal{L}_\xi$, obtained by
linearing \eqref{equv} at $W(\cdot,\xi)$, has a first eigenvalue that is real and simple,
applying a projection method in the spirit of
 \cite{ReynWard95b}, we determine the equation (see Section \ref{sec:compressible})
\begin{equation*}
	\frac{d\xi}{dt}=-\frac{\psi(\xi,\xi)}{\phi(\xi,\xi)}\frac{\kappa_+(\xi)-\kappa_-(\xi)}{u_+-u_-}.
\end{equation*}
where $\bigl(\phi(\cdot,\xi),\psi(\cdot,\xi)\bigr)$ is the first eigenfunction of the adjoint
operator $\mathcal{L}_\xi^\ast$ and the functions $\kappa_\pm$ are implicitly defined by
\begin{equation*}
	\int_{u_\ast}^{u_\pm} \frac{v_\ast\,\nu(s)}{s^2(\kappa_\pm-F(s,v_\ast))}\,ds
	=\pm\frac{\ell\mp\xi}{\varepsilon}
\end{equation*}
where $u_\ast$ is the value determined by the condition $P'(u_\ast)u_\ast^2=v_\ast^2$.
The difference function $\xi\mapsto \kappa_+(\xi)-\kappa_-(\xi)$ is continuous,
monotone increasing, diverges at $\pm\infty$ as $\xi\to\pm\ell$; hence, the above
differential equation possesses a single equilibrium point $\xi_\ast$, corresponding
to the unique stationary solution of the problem.

For small $\varepsilon$, the elements of the approximate manifolds tend to a piecewise
constant configuration, with a single jump located at $\xi$.
Thus, it is possible to determine the leading term in the expression for the eigenfuntion
$(\phi,\psi)$ of the adjoint operator $\mathcal{L}_\xi^\ast$ and to obtain a new version
of the motion equation
\begin{equation*}
	\frac{d\xi}{dt}=-\frac{(\xi+\ell)\lambda_1(\xi)}{\partial_u F(u_-,v_\ast)}
		\frac{\kappa_+(\xi)-\kappa_-(\xi)}{u_+-u_-}.
\end{equation*}
where $\lambda_1(\xi)$ is the first eigenvalue of the operator $\mathcal{L}_\xi$.
Let us stress that, at such step, the choice of the boundary conditions is particularly
relevant, since it determines the specific structure of the eigenfunction $(\phi,\psi)$.

As it should be, the behavior of $\lambda_1$ for small $\varepsilon$ plays
a crucial r\^ole (details are given in Section \ref{sec:compressible}).
Extrapolating from \cite{KreiKreiLore08}, that concerns with general systems of 
conservation laws with a (non-physical) second-order parabolic term, and by
the quoted results on asymptotic stability of steady waves on the whole real line,
we expect that $\lambda_1$ is (exponentially) small and negative.
We are not aware of any availabe rigorous result in this sense at the present time.

Assuming that such eigenvalue stayes bounded in the limit $\varepsilon\to 0^+$,
the speed of motion along the approximate manifold, is determined by the functions
$\kappa_\pm$, whose leading terms can be determined starting from their definitions,
obtaining the final form for the motion's equation
\begin{equation*}
	\frac{d\xi}{dt}=-\biggl\{\frac{\partial_u F^+}{\partial_u F^-}\frac{u_+-u_\ast}{u_+-u_-}
		\exp\left(-\frac{\partial_u F^+}{\partial_u G^+}\frac{\ell-\xi}{\varepsilon}\right)
			+\frac{u_\ast-u_-}{u_+-u_-}
	\exp\left(\frac{\partial_u F^-}{\partial_u G^-}\frac{\xi+\ell}{\varepsilon}\right)\biggr\}
		(\xi+\ell)\lambda_1(\xi)
\end{equation*}
where $\partial_u G(u,v)=v\nu(u)/u^2$ and the upper scripts ${}^\pm$ indicate that the
function is calculated at $(u_\pm,v_\ast)$.
In particular, for small $\varepsilon$, the motion is exponentially slow.
The stability of the equilibrium point $\xi_\ast$ is still encoded in the sign of the 
first eigenvalue $\lambda_1$ that the present analysis is not able to reveal.

This paper follows the research line on metastable behaviors for conservation laws
widely explored in the last decades.
The first contribution has been the pioneering article \cite{KreiKrei86} concerning the
analysis of the scalar Burgers equation, that has been also the subject of 
\cite{ReynWard95b} (based on the use of  {\it projection method} and {\it WKB expansions})
and \cite{LafoOMal95a} (standing on an adapted version of the {\it method of matched
asymptotics expansion}).
A rigorous analysis has been performed in \cite{deGrKara98, deGrKara01}), where 
the one-parameter family of reference functions is chosen as a family of traveling wave
solutions to the viscous equation satisfying the boundary conditions and with non-zero 
velocity. 
Slow motion for the viscous Burgers equation in unbounded domains has been also 
considered in literature: the case of the half-line $(0,+\infty)$ has been treated in
\cite{Shih95, LiuYu97, Nish01}; while the case of whole real line  has been examined
in \cite{KimTzav01, KimNi02, BeckWayn09} (with emphasis on  the generation of $N-$wave
like structures and their evolution towards nonlinear diffusion waves).

Despite of the wide number of contributions to the stability of traveling profiles in the whole
real line, results relative to slow motion and metastable behavior in the case of systems
of conservation laws appear to be rare. 
We are only aware of \cite{HubeSerr96} (that uses asymptotic expansions to deal with 
systems of conservation laws, with model examples being the Navier-Stokes equations
of compressible viscous heat  conductive fluid and the Keyfitz-Kranzer system, arising in elasticity),
\cite{KreiKreiLore08} (that deals with the problem of proving convergence to a stationary solution
for a system of conservation laws with viscosity, with an approach based on the analysis of the
linearized operator at the steady state)
and \cite{BJRZ11} (that addresses to the Saint-Venant equations for shallow water and, precisely,
the phenomenon of formation of roll-waves, by means of a combination of analytical techniques
and numerical results).

In this respect, we consider our contribution, even if mainly based on formal arguments, 
original and hopefully stimulating for people working in the area of dynamical properties of
solutions to systems of conservation laws.
Specifically, it seems that the equation we propose for the motion of the transition layer $\xi$
(that is the slow dynamics along the approximate equilibrium manifold) is the first attempt
in this direction for isentropic Navier--Stokes equations for compressible fluids.

The article is divided into three more Sections.
In Section 2, we present the general procedure to derive formal equation for the motion along
an approximate equilibirum manifold in the case of general hyperbolic--parabolic systems.
For pedagogical reasons, we also show how the approach simplifies in the case of scalar
Burgers equation.
Section 3 is the heart of the paper. 
It deals with isentropic Navier--Stokes equation for compressible fluids, written in the
form \eqref{equv}.
After recasting the admissibility condition for entropic jumps in the unviscous case,
we build up the approximate equilibrium manifold, working with time-independent
solutions and matching them at $x=\xi$ by means of appropriate transmission condition.
Using such special approximate solutions, we are able to determine an equation describing
the evolution along the manifold.
In order to determine an explicit expression for the ratio $\psi(\xi,\xi)/\phi(\xi,\xi)$ appearing
in such equation for the motion along the manifold, we analyze the eigenvalue problem for
the adjoint operator and we deduce an approximated version of its solutions for the
regime $\varepsilon$ small, by approximating the element of the manifold
to be piecewise constant functions with a single jump located at $\xi$.
Finally, in order to get an ultimate version of the motion's equation,
we determine the leading term for the functions $\kappa_\pm$ in the limit
$\varepsilon\to 0^+$, 
Section 4 contains the conclusions and it is mainly dedicated to propose a number of
eventual research direction motivated by the present work.

\section{Scalar reduced dynamics for hyperbolic--parabolic systems}\label{sec:scalar}

As a first step, we present the strategy to obtain approximate equation for the motion
along an approximate manifold of solutions to a general hyperbolic-parabolic system
having the form
\begin{equation}\label{hyppar}
	\frac{\partial w}{\partial t}+\frac{\partial}{\partial x}\left\{f(w)
		-\varepsilon b(w)\,\frac{\partial w}{\partial x}\right\}=0,
\end{equation}
where $w=w(x,t)\in\mathbb{R}^n$, $x\in I$, $t>0$.
Our target is to apply such approach to the case of isentropic Navier--Stokes equation.
Thus, for the sake of simplicity, we do not state precise assumptions on the
structure functions $f$ and $b$ and we proceed in a purely formal way.
In any case, we expect the procedure to be meaningful for the usual class of 
hyperbolic--parabolic systems considered in the recent literature
(see \cite{ShizKawa85} and descendants).

Given $T>0$, we consider the initial-boundary value problem for \eqref{hyppar}
determined by the conditions
\begin{equation}\label{initialbdary}
	w \bigr|_{t=0}=w_0
\end{equation}
complemented with appropriate boundary conditions.

Given an open interval $J$ and a one-parameter family of functions
\begin{equation*}
	\{W(\cdot;\xi)\in [H^1(I)]^n\,:\,\xi\in J\}
\end{equation*}
satisfying the boundary conditions,
let $\xi\mapsto R^\varepsilon(\cdot;\xi)$ be the distribution-valued map defined by
\begin{equation}\label{rest}
	\langle R(\cdot;\xi),\varphi\rangle:=-\int_{I} \Bigl\{\varepsilon b(W)
		\frac{\partial W}{\partial x}-f(W)\Bigr\}\cdot\frac{d\Phi}{dx}\,dx
\end{equation}
for any continuously differentiable function $\Phi\,:\,I\to\mathbb{R}^n$ with $\Phi(\pm\ell)=0$.
In what follows, we call $R(\cdot;\xi)$ the {\sf residual} of $W(\cdot;\xi)$ with respect to
equation \eqref{hyppar}.
The family $\{W(\cdot;\xi)\}$ is considered as an {\it approximate invariant manifold} for \eqref{hyppar},
in the sense that the residuals $R(\cdot;\xi)$ vanishes as $\varepsilon\to 0^+$ (in a sense to be made
precise) for any $\xi\in J$.

Next, we look for solutions to \eqref{hyppar}--\eqref{initialbdary} in the form 
\begin{equation*}
	w(\cdot,t)=W(\cdot;\xi(t))+z(\cdot,t)
\end{equation*} 
with unknown $\xi=\xi(t)$  and $z=z(\cdot,t)$ to be determined.
Substituting into \eqref{hyppar} and disregarding the nonlinear terms in $v$,
we obtain an approximated equation for the perturbation $z$
\begin{equation}\label{eqz}
	\frac{\partial z}{\partial t}=\mathcal{L}_\xi z+R(\cdot;\xi)
		-\frac{\partial W}{\partial\xi}(\cdot;\xi)\,\frac{d\xi}{dt}
\end{equation}
where $R(\cdot;\xi)$ is the distribution defined in \eqref{rest} and $\mathcal{L}_\xi$
is the linearized operator at $W(\cdot;\xi)$, i.e.
\begin{equation*}
	\mathcal{L}_\xi z:=\frac{\partial}{\partial x}\left\{\varepsilon b(W)\frac{\partial z}{\partial x}
	+db(W)\,z\,\frac{\partial W}{\partial x}-df(W)z\right\}
\end{equation*}
where we use the notation
\begin{equation*}
	(db(W)zw)_i:=\sum_{j,k} \frac{\partial b_{ij}}{\partial w_k}(W) z_k w_j
	\qquad\qquad\textrm{with } b=(b_{ij}).
\end{equation*}
Next, let us assume that, for any $\xi\in I$, the operator $\mathcal{L}_\xi$ has a first eigenvalue
$\lambda_1(\xi)$ that is real and simple.
Let $\mathcal{L}_\xi^{\ast}$ be the adjoint operator of $\mathcal{L}_\xi$
\begin{equation*}
	\mathcal{L}_\xi^{\ast} z:=\frac{\partial}{\partial x}\left\{\varepsilon b(W)^t\frac{\partial z}{\partial x}\right\}
	-\left(Db(W)\,\frac{\partial W}{\partial x}-df(W)^t \right)\frac{\partial z}{\partial x}
\end{equation*}
where ${}^{t}$ denote the transpose and 
\begin{equation*}
	(Db(W)wz)_k:=\sum_{i,j} \frac{\partial b_{ij}}{\partial w_k}(W) z_i w_j
	\qquad\qquad\textrm{with } b=(b_{ij}).
\end{equation*}
Given $\ell>0$, set $I:=(-\ell,\ell)$ and
\begin{equation*}
	\langle u, v\rangle:=\int_{-\ell}^{\ell} u(x)\cdot v(x)\,dx
	\qquad\qquad u,v\in [L^2(I)]^n,
\end{equation*}
where $\cdot$ denotes the usual scalar product in $\mathbb{R}^n$.
Denoting by $\omega_1=\omega_1(\cdot;\xi)$ an eigenfunction of the adjoint operator 
relative to the first eigenvalue and setting 
\begin{equation*}
	z_1=z_1(\xi;t):=\langle \omega_1(\cdot;\xi),z(\cdot,t)\rangle, 
\end{equation*}
we determine a differential equation for the function $t\mapsto \xi(t)$ by imposing
that  the component $z_1$  is identically zero, that is
\begin{equation*}
	\frac{d}{dt} \langle \omega_1(\cdot;\xi(t)), z(\cdot,t) \rangle =0
	\qquad\textrm{and}\qquad
	\langle \omega_1(\cdot;\xi_0), z(\cdot,0))\rangle=0.
\end{equation*}
Using equation \eqref{eqz}, we infer
\begin{equation*}
	\langle \omega_1(\cdot;\xi),\mathcal{L}_\xi z+ R(\cdot;\xi)
		-\frac{\partial W}{\partial \xi}(\cdot;\xi)\frac{d\xi}{dt}\rangle 
		+ \langle \frac{\partial \omega_1}{\partial \xi}\frac{d\xi}{dt},z \rangle =0
\end{equation*}
Since $\langle \omega_1, {\mathcal L}_{\xi}z \rangle= \lambda_1\langle \omega_1, z \rangle$, 
we obtain a scalar differential equation for the variable $\xi$, 
the latter equation can be rewritten as
\begin{equation*}
		\left\{\langle \omega_1,\frac{\partial W}{\partial \xi}\rangle 
		- \langle \frac{\partial \omega_1}{\partial \xi},z \rangle\right\}
		\frac{d\xi}{dt}=\langle \omega_1,R \rangle
\end{equation*}
to be considered together with the condition on the initial datum $\xi_0$
\begin{equation*}
\langle \omega_1(\cdot;\xi_0), z(\cdot,0) \rangle =0.
\end{equation*}
Neclecting the second term in the coefficient of the derivative of $\xi$,
we end up with an (approximated) equation for the motion along the
manifold $\{W(\cdot;\xi)\}$
\begin{equation}\label{approxedo}
	\langle \omega_1(\cdot;\xi),\frac{\partial W}{\partial\xi}(\cdot;\xi)\rangle\frac{d\xi}{dt}
	=\langle \omega_1(\cdot;\xi),R(\cdot;\xi) \rangle.
\end{equation}
Our next effort is to determine a practical version of the above equation in the limiting regime
$\varepsilon\to 0^+$ and for the specific case of hyperbolic--parabolic systems \eqref{hyppar}.

First of all, given $\xi\in I=(-\ell,\ell)$, we require the function $W(\cdot;\xi)$ to converge
in the limit $\varepsilon\to 0^+$ to the step function jumping at $x=\xi$ from $w_-$ to $w_+$
\begin{equation}\label{limitW}
	\lim_{\varepsilon\to 0^+} W(x,\xi)=w_-\chi_{{}_{(-\ell,\xi)}}(x)+w_+\chi_{{}_{(\xi,+\ell)}}(x)
\end{equation}
If the limit is in the sense of distributions, approximating the increment ratio for $h>0$ by
\begin{equation*}
	\frac{W(x;\xi+h)-W(x,\xi)}{h}\approx -\frac{1}{h}[w]\chi_{{}_{(\xi,\xi+h)}}(x)
\end{equation*}
we infer the asymptotic representation
\begin{equation*}
	\frac{\partial W}{\partial\xi}(\cdot;\xi)=-[w]\delta_{\xi}(\cdot)+o(1)
	\qquad\textrm{as }\varepsilon\to 0^+,
\end{equation*}
where $[w]:=w_+-w_-$ and $\delta_{\xi}$ is the Dirac distribution concentrated at $\xi$.
Thus, for small $\varepsilon$, there hold
\begin{equation*}
	\langle \omega_1(\cdot;\xi),\frac{\partial W}{\partial\xi}(\cdot;\xi)\rangle
	=-\omega_1(\xi;\xi)\cdot [w].
\end{equation*}
Finally, assuming $W$ to be chosen so that $R$ is a Dirac distribution concentrated at $\xi$,
\begin{equation}\label{deltaresidual}
	R(x,\xi)=r(\xi)\delta_{\xi}(x)
\end{equation}
for some function $r=r(\xi)$, we deduce our final expression for the reduced dynamics
\eqref{approxedo} along the manifold $\{W(\cdot;\xi)\}$
\begin{equation}\label{rededo}
	\frac{d\xi}{dt}=\theta(\xi):= -\frac{\omega_1(\xi;\xi)\cdot r(\xi)}{\omega_1(\xi;\xi)\cdot [w]}.
\end{equation}
In order to make equation \eqref{rededo}, it is necessary to determine the specific
expression of the function $r$, describing the residual of $W(\cdot;\xi)$, and 
the ratio between the components of the first eigenfunction $\omega_1$ of the 
adjoint operator in the direction of $r$ and in the direction of the jump $[w]$.

\subsection*{Viscous Burgers equation}\label{subsec:expected}

Let us consider the case of the scalar Burgers equation with viscosity
\begin{equation}\label{burgers}
	\frac{\partial w}{\partial t}+\frac{\partial}{\partial x}\left\{\frac12 w^2
		-\varepsilon \frac{\partial w}{\partial x}\right\}=0.
\end{equation}
with boundary conditions
\begin{equation*}
	w(-\ell,t)=\bar w,\qquad w(+\ell,t)=-\bar w
\end{equation*}
for some given $\bar w>0$.
In this case $[w]=-2\overline{w}$ and equation \eqref{rededo} reduces to
\begin{equation}\label{slowburgers}
	\frac{d\xi}{dt}=\frac{r(\xi)}{2\overline{w}},
\end{equation}
where $r(\xi)$ is the residual of $W(\cdot;\xi)$. 

In order to satisfy the requirement \eqref{deltaresidual},
we consider a specific approximate invariant manifold:
for $\xi\in I$, we build $W(x,\xi)$ by matching two steady
states of the equation with appropriate boundary conditions.
Namely, we set
\begin{equation*}
	W(x,\xi):=\left\{\begin{aligned}
		&W_-(x,\xi) 	&\qquad &	-\ell<x<\xi<\ell \\
		&W_+(x,\xi)	&\qquad &-\ell<\xi<x<\ell ,
           \end{aligned}\right.
\end{equation*}
where $W_\pm$ are steady state of \eqref{burgers} in $(-\ell,\xi)$ and $(\xi,\ell)$,
such that
\begin{equation*}
	W_-(-\ell;\xi)=\overline{w},\qquad W_-(\xi;\xi)=W_+(\xi;\xi)=0,\qquad W_+(\ell;\xi)=-\overline{w}.
\end{equation*}
Functions $W_\pm$ can be expressed by means of an implicit formula.
For $\overline{w}>0$, let us set
\begin{equation*}
	\Sigma:=\left\{(w,\kappa)\,:\,w\in(-\overline{w},\overline{w}),\; \kappa>w^2/2\right\}.
\end{equation*}
Then,  defining the function $\Gamma=\Gamma(w,\kappa)$ with $(w,\kappa)\in\Sigma$ by 
\begin{equation*}
	\Gamma(w;\kappa):=\int_{0}^{w} \frac{ds}{\kappa-s^2/2}
	=\sqrt{\frac{2}{\kappa}}\tanh^{-1}\left(\frac{w}{\sqrt{2\kappa}}\right),
\end{equation*}
functions $W_\pm$ are implictly given by
\begin{equation}\label{implicitW}
	\varepsilon\,\Gamma(W_\pm(x,\xi),\kappa_\pm)=\xi-x.
\end{equation}
where the values $\kappa_\pm=\kappa_\pm(\xi)$ are uniquely determined by the conditions
\begin{equation}\label{defkappapm}
	\varepsilon\,\Gamma(\pm\overline{w},\kappa_\mp)=\xi\pm\ell.
\end{equation}
Since $W_-$ and $W_+$ are steady states of \eqref{burgers} in $(-\ell,\xi)$ and $(\xi,\ell)$,
respectively, we deduce, integrating by parts, that the residual $R$ is 
\begin{equation*}
	\begin{aligned}
	\langle R(\cdot;\xi),\Phi\rangle&=-\int_{-\ell}^{\xi} \Bigl\{\varepsilon 
			\frac{\partial W_-}{\partial x}-\frac12 W^2_-\Bigr\}\frac{d\Phi}{dx}\,dx
			-\int_{\xi}^{\ell} \Bigl\{\varepsilon 
			\frac{\partial W_+}{\partial x}-\frac12 W^2_+\Bigr\}\frac{d\Phi}{dx}\,dx\\
		&=-\varepsilon\frac{\partial W_-}{\partial x}(\xi)\Phi(\xi)
			+\varepsilon\frac{\partial W_+}{\partial x}(\xi)\Phi(\xi)
			=\varepsilon\left[\frac{\partial W}{\partial x}\right]_{\xi}\,\Phi(\xi)
	\end{aligned}
\end{equation*}
Differentiating \eqref{implicitW} with respect to $x$, we get
\begin{equation}
	\varepsilon\frac{\partial W_\pm}{\partial x}(x,\xi)
	=\frac{1}{2}W_\pm^2(x,\xi)-\kappa_\pm,
\end{equation}
thus, in the notation of \eqref{deltaresidual}, there holds
\begin{equation*}
	r(\xi)=\varepsilon\left[\frac{\partial W}{\partial x}\right]_{\xi}
		=\varepsilon\frac{\partial W_+}{\partial x}(\xi;\xi)-\varepsilon\frac{\partial W_-}{\partial x}(\xi;\xi)
		=\kappa_-(\xi)-\kappa_+(\xi),
\end{equation*}
giving an ``almost explicit'' expression for \eqref{slowburgers}.

Since we are considering the regime $\varepsilon\to 0^+$, it is possible to
approximate the formulas defining $\kappa_\pm$ and obtain a simpler o.d.e.
describing the slow dynamics.
Handling the explicit expression for the function $\Gamma$,
conditions \eqref{defkappapm} become
\begin{equation*}
	\sqrt{2\kappa_\mp}\tanh\left\{\frac{\ell\pm\xi}{\varepsilon}
		\sqrt{\frac{\kappa_\mp}{2}}\right\}=\overline{w}.
\end{equation*}
In particular, in the limit $\varepsilon\to 0^+$, we get
\begin{equation*}
	\lim_{\varepsilon\to 0^+} \kappa_\pm(\xi)=\frac12\overline{w}^2.
\end{equation*}
Thus, as $\varepsilon\to 0^+$, there approximately holds
\begin{equation*}
	\kappa_\mp\approx \frac{\overline{w}^2}{2\tanh\left\{\overline{w}(\ell\pm\xi)/2\varepsilon\right\}^2}
		\approx \frac{1}{2}\overline{w}^2\left\{1
			+4\exp\left(-\frac{\overline{w}}{\varepsilon}(\ell\pm\xi)\right)\right\}
\end{equation*}
Collecting, we end up with the equation
\begin{equation*}
	\frac{d\xi}{dt}\approx \overline{w}
		\left\{\exp\left(-\frac{\overline{w}}{\varepsilon}(\ell+\xi)\right)
			-\exp\left(-\frac{\overline{w}}{\varepsilon}(\ell-\xi)\right)\right\}
\end{equation*}
corresponding to the formula determined in \cite{ReynWard95b}.

\section{Compressible isentropic Navier--Stokes equations}\label{sec:compressible}

Given smooth functions $P=P(u)$ and $\nu=\nu(u)$,
let us consider the hyperbolic-parabolic system for compressible
isentropic fluids
\begin{equation}\label{ns}
	\frac{\partial u}{\partial t}+\frac{\partial v}{\partial x}=0,\qquad
	\frac{\partial v}{\partial t}+\frac{\partial}{\partial x}\left\{\frac{v^2}{u}+ P(u)
			-\varepsilon\nu(u)\frac{\partial}{\partial x}\left(\frac{v}{u}\right)\right\}=0.
\end{equation}
where the pressure $P$ is such that $P'(u)>0$ and $P''(u)>0$ and the
viscosity $\nu$ is such that $\nu(u)>0$ for any $u$ under consideration.

System \eqref{ns} is considered for $x\in(-\ell,\ell)$ together with the boundary conditions
\begin{equation}\label{bdaryuv2}
	v_\pm\,u(\pm\ell)-u_\pm\,v(\pm\ell)=0,\qquad	v(-\ell)=v_-
\end{equation}
for some $u_\pm, v_\pm$ with the value $v_-$ being strictly positive.

\subsection*{Admissible jumps in the vanishing viscosity limit}\label{subsec:admissible}

As a first step, let us consider the limiting unviscous regime $\varepsilon \to 0^+$.
Putting formally $\varepsilon=0$ in \eqref{ns},
we obtain the hyperbolic system for unviscous isentropic fluids
\begin{equation}\label{euler}
	\frac{\partial u}{\partial t}+\frac{\partial v}{\partial x}=0,\qquad
	\frac{\partial v}{\partial t}+\frac{\partial}{\partial x}\left\{\frac{v^2}{u}+ P(u)\right\}=0.
\end{equation}
Physical solutions are weak solutions to \eqref{euler} satisfying appropriate jump conditions
determined by the couple entropy/entropy flux
\begin{equation*}
	\mathcal E(u,v):=\frac{v^2}{2u}+\Pi(u), \qquad
	\mathcal Q(u,v):= \frac{v^3}{2u^2}+\Pi'(u)v.
\end{equation*}
where $\Pi$ is such that $\Pi''(u)=P'(u)/u$.
In particular, given $u_\pm >0$, $v_\pm\in\mathbb{R}$ and $c\in\mathbb{R}$, the function
\begin{equation}\label{jump}
	(u,v)(x,t):=(u_-,v_-)\chi_{{}_{(-\infty,ct)}}(x)+(u_+,v_+)\chi_{{}_{(ct,+\infty)}}(x)
\end{equation}
(where $\chi_{{}_{I}}$ is the characteristic function of the set $I$) is an {\sf admissible solution}
to \eqref{euler} if it is a weak solution and it satisfies the inequality
\begin{equation}\label{PNEineq}
	\partial_t E+ \partial_x\mathcal Q \leq 0
\end{equation}
in the sense of distributions.
Thanks to the galileian invariance, we may assume without loss of generality, the 
speed $c$ to be zero.
Thus, the requirement of being a weak solution translates in the classical
{\sf Rankine-Hugoniot conditions}, that read as
\begin{equation}\label{rh}
	\left[v\right]_0=0, \qquad \left[\frac{v^2}{u}+P(u)\right]_0=0
\end{equation}
where $[g]_\xi:=g_+-g_-$ denotes the jump of the function $g$ at $x=\xi$.
Similarly, the entropy condition \eqref{PNEineq} reads as
\begin{equation}\label{entropy}
	[\mathcal Q]_0=\left[\frac{v^3}{2u^2}+\Pi'(u)v\right]_0\leq 0.
\end{equation}
Conditions \eqref{rh}--\eqref{entropy} select the possible couples $(w_-,w_+)$
such that  jump solution \eqref{jump} defines an admissible solution to \eqref{euler}.
The admissible couples can be explicitly characterized.

\begin{lemma}\label{lemma:propertyF}
Given $P\in C^2([0,+\infty))$ such that $P'(u), P''(u)>0$ for any $u\geq 0$, let us set
\begin{equation}\label{defF}
 	F(u,v):=\frac{v^2}{u}+P(u)\qquad\qquad u>0, v\in\mathbb{R}.
\end{equation}
Then, for any $v_1>0$, there exists $f(v_1)>0$ such that for any $v_2>f(v_1)$
equation $F(u,v_1)=v_2$ possesses exactly two solutions $u_\pm:=u_\pm(v_1,v_2)$.
Moreover, the function $f$ is strictly increasing and convex and such that
$f(0)=P(0)$, $f'(0)=2\sqrt{P'(0)}$.
\end{lemma}

\begin{proof} There hold
\begin{equation*}
	\frac{\partial F}{\partial u}(u,v)=-\frac{v^2}{u^2}+P'(u),\qquad
	\frac{\partial^2 F}{\partial u^2}(u,v)=\frac{2v^2}{u^3}+P''(u)
\end{equation*}
Hence, for any $v>0$, the function $F(\cdot,v)$ has a single absolute minimum
point $u_\ast=u_\ast(v)$ uniquely defined by the implicit relation
\begin{equation*}
	P'(u_\ast)u_\ast^2=v^2.
\end{equation*}
Then, the function $f=f(v)$ is defined by
\begin{equation*}
	f(v):=\min_{u>0} F(u,v)=\frac{v^2}{u_\ast}+P(u_\ast)
		=P'(u_\ast)u_\ast+P(u_\ast).
\end{equation*}
Moreover, differentiating we deduce
\begin{equation*}
	\frac{df}{dv}=\left\{P''(u_\ast)u_\ast+2P'(u_\ast)\right\}\,\frac{du_\ast}{dv}
		=\frac{2v}{u_\ast}>0
\end{equation*}
and
\begin{equation*}
	\frac{d^2f}{dv^2}=\frac{2P''}{P''u_\ast+2P'}>0,
\end{equation*}
showing the stated properties of the function $\phi$.
\end{proof}

\begin{proposition}
Let $P\in C^2([0,+\infty))$ be such that $P'(u), P''(u)>0$ for any $u\geq 0$.
For any couple $(v_1,v_2)$ with $v_1>0$ and $v_2>f(v_1)$ there
exists unique $(u_\pm,v_\pm)$ such that
\begin{equation*}
	v_\pm=v_1,\qquad
	F(u_\pm,v_1)=\frac{v_1^2}{u_\pm}+P(u_\pm)=v_2,\qquad
	u_-<u_+
\end{equation*}
such that the function \eqref{jump} is a weak solution to
\eqref{euler} satisfying condition \eqref{entropy}.
\end{proposition}

\begin{proof}
Couples $(u_\pm,v_\pm)$ connected by a single entropic stationary jump are
such that $v_-$ and $v_+$ are equal with common value denoted by $v_\ast$ and
\begin{equation*}
	F(u_-,v_\ast)=F(u_+,v_\ast).
\end{equation*}
The properties of the function $F$ shows that, given $v_\ast$ there is
a single couple of values for which the above relation is satisfied.
It only remains to analyze condition \eqref{entropy}.

For $v_1>0$, the entropy condition becomes
\begin{equation}\label{entro}
	\Lambda(u_+,v_1):=\frac{v_1^2}{u_+^2}+2\Pi'(u_+)
		\leq \frac{v_1^2}{u_-^2}+2\Pi'(u_-)=\Lambda(u_-,v_1).
\end{equation}
Since there holds
\begin{equation*}
	\frac{\partial\Lambda}{\partial u}(u,v_1)=-\frac{2v_1^2}{u^3}+2\Pi''(u)
		=\frac{2}{u}\frac{\partial F}{\partial u}(u,v_1),
\end{equation*}
the function $\Lambda$ has the same monotonicity of $F$ with growth rate that
decreases when $u$ increases with respect to the one of $F$.
Thus, condition \eqref{entro} is satisfied if and only if $u_+\geq u_-$.
\end{proof}

As an example, in the power-law case $P(u)=\kappa\,u^{\alpha+1}/(\alpha+1)$ with 
$\alpha>0$, there holds
\begin{equation*}
	u_\ast=\kappa^{-\frac{1}{\alpha+2}}\,v^{\frac{2}{\alpha+2}},
	\qquad
	f(v)=\frac{\alpha+2}{\alpha+1}\,\kappa^{\frac{1}{\alpha+2}}\,v^{\frac{2(\alpha+1)}{\alpha+2}}
\end{equation*}
Note that the function $f$ is not differentiable two times at $u=0$.

\subsection*{Approximate invariant manifold}\label{subsec:approximate}

Next, we build a one-parameter family of functions $W=W(\cdot;\xi)$ for $\xi\in (-\ell,\ell)$
forming an approximate invariant manifold for \eqref{ns} and converging as $\varepsilon\to 0$ to
\begin{equation}\label{hypfamily}
	W_{{}_{\textrm{hyp}}}(x,\xi)
	=(u_-,v_\ast)\chi_{{}_{(-\ell,\xi)}}(x)+(u_+,v_\ast)\chi_{{}_{(\xi,+\ell)}}(x)
\end{equation}
Recalling the general procedure presented in Section \ref{sec:scalar}, we want to choose $W$
so that the residual $R(\xi)$ is a delta distribution concentrated at $\xi$.
Thus, given $\xi\in I$, we opt for defining $W$ by matching at $\xi\in I$ at the state
$(u_\ast,v_\ast)$ the two stationary solutions of \eqref{ns} in $(-\ell,\xi)$ and $(\xi,\ell)$.
Precisely, we denote by $W_-=(U_-,V_-)$ and $W_+=(U_+,V_+)$ the solutions to
\begin{equation}\label{steadySV}
	\frac{dv}{dx}=0,\qquad
	\frac{d}{dx}\left\{\frac{v^2}{u}+ P(u)
			-\varepsilon\nu(u)\frac{d}{dx}\left(\frac{v}{u}\right)\right\}=0,
\end{equation}
in $(-\ell,\xi)$ and in $(\xi,\ell)$ respectively, satisfying the boundary conditions
\begin{equation*}
	\begin{aligned}
		W_-(-\ell;\xi)&=(u_-,v_\ast), 	&\qquad	W_-(\xi;\xi)&=(u_\ast,v_\ast),\\
		W_+(\xi;\xi)&=(u_\ast,v_\ast),	&\qquad	W_+(\ell;\xi)&=(u_+,v_\ast),
	\end{aligned}
\end{equation*}
where $u_\ast$ is the absolute minimum point of the function $F(\cdot,v_\ast)$,
i.e. is uniquely determined by the requirement $P'(u_\ast)u_\ast^2=v_\ast^2$
(see Lemma \ref{lemma:propertyF}).

Solutions to \eqref{steadySV} solve
\begin{equation}\label{steadySV2}
	v=v_1,\qquad
	\varepsilon\,\frac{\partial G}{\partial u}(u,v_1)\,\frac{du}{dx}=\kappa-F(u,v_1)
\end{equation}
where $v_1, \kappa$ are integration constants, 
the function $F$ is defined in \eqref{defF} and 
the function $G$ is given by
\begin{equation*}
	G(u,v):=v\int_{\bar u}^{u} \frac{\nu(s)}{s^2}\,ds,
\end{equation*}
for some fixed $\bar u>0$.
From now on, let us consider $v_1>0$.

For $F$ defined in \eqref{defF} and given $v_\ast$, let us set
\begin{equation*}
	\Sigma:=\{(u,\kappa)\,:\,u\in(u_-,u_+),\quad \kappa>F(u,v_\ast)\}
\end{equation*}
and define the function
\begin{equation*}
	\Gamma(u,\kappa):=\int_{u_\ast}^{u} \frac{\partial_u G(s,v_\ast)}{\kappa-F(s,v_\ast)}\,ds
	\qquad\qquad \textrm{for}\;(u,\kappa)\in\Sigma.
\end{equation*}
Then, solutions to \eqref{steadySV} satisfying the condition $u(\xi)=u_\ast$ are implicitly defined by
\begin{equation*}
	v(x)=v_\ast,\qquad \varepsilon\,\Gamma(u(x),\kappa)=x-\xi
\end{equation*}
Denoting by $u_\pm(\kappa)$ the two solution of $F(u,v_\ast)=\kappa$
(see Lemma \ref{lemma:propertyF}), the function $\Gamma$ is such that
\begin{equation*}
	\begin{aligned}
	&\Gamma(u,+\infty)=0,\qquad \Gamma(u_-(\kappa),\kappa)=-\infty,
		\qquad 
		\Gamma(u_+(\kappa),\kappa)=+\infty\\
	&\Gamma(u,\cdot)\;\textrm{is increasing if }u<u_\ast,\qquad
		\Gamma(u,\cdot)\;\textrm{is decreasing if }u>u_\ast.
	\end{aligned}
\end{equation*}
As a consequence, for any $\xi\in(-\ell,\ell)$ there exist (unique)
$\kappa_\pm=\kappa_\pm(\xi)\in(F(u_\pm,v_\ast),+\infty)$ such that
\begin{equation*}
	\varepsilon\,\Gamma(u_-,\kappa_-)=-\ell-\xi
		\qquad\textrm{and}\qquad
	\varepsilon\,\Gamma(u_+,\kappa_+)=\ell-\xi.
\end{equation*}
Correspondingly, we set
\begin{equation*}
	W(x,\xi)=\left\{\begin{aligned}
		&(U_-(x,\xi),v_\ast)	&\qquad &	-\ell<x<\xi<\ell \\
		&(U_+(x,\xi),v_\ast)	&\qquad &-\ell<\xi<x<\ell ,
           \end{aligned}\right.
\end{equation*}
where functions $U_\pm$ are implicitly given by
\begin{equation}\label{implicita}
	\varepsilon\,\Gamma(U_\pm(x,\xi),\kappa_\pm)=x-\xi.
\end{equation}
Next, let us calculate the residual of $W$: for any test function $\Phi=(\varphi_1,\varphi_2)$,
there holds
\begin{equation*}
	\begin{aligned}
	&\langle R(\cdot;\xi),\Phi\rangle
		=\int_{I} \Bigl\{F(U,v_\ast)
			-\varepsilon\frac{\partial G}{\partial u}(U,v_\ast)\frac{\partial U}{\partial x}\Bigr\}\frac{d\varphi_2}{dx}\,dx\\
		&\qquad=\left\{\left(F(U_-,v_\ast))-\varepsilon\frac{\partial G}{\partial u}(U_-,v_\ast)\frac{\partial U_-}{\partial x}\right)
			-\left(F(U_+,v_\ast)-\varepsilon\frac{\partial G}{\partial u}(U_+,v_\ast)
				\frac{\partial U_+}{\partial x}\right)\right\}\biggr|_{x=\xi}\varphi_2(\xi)\\
		&\qquad=\varepsilon 
			\left[\frac{\partial G}{\partial u}(U,v_\ast)\frac{\partial U}{\partial x}\right]_{\xi}\varphi_2(\xi)
	\end{aligned}
\end{equation*}
Differentiating \eqref{implicita} with respect to $x$ we infer
\begin{equation*}
	\varepsilon \frac{\partial G}{\partial u}(U_\pm(x,\xi),v_\ast)\frac{\partial U_\pm}{\partial x}(x,\xi)
		=\kappa_\pm-F(U_\pm(x,\xi),v_\ast)
\end{equation*}
so that, with $r=r(\xi)$ as in \eqref{deltaresidual}, we obtain
\begin{equation*}
	r(\xi)=\left(0,\kappa_+(\xi)-\kappa_-(\xi)\right)
\end{equation*}
The functions $\xi\mapsto \kappa_\pm(\xi)$ are the inverse of the relations
\begin{equation*}
	\xi=-\ell-\varepsilon\,\Gamma(u_-,\kappa_-)
		\qquad\textrm{and}\qquad
	\xi=\ell-\varepsilon\,\Gamma(u_+,\kappa_+),
\end{equation*}
thus, as a consequence of the properties of function $\Gamma$, the difference function
$\xi\mapsto \kappa_+(\xi)-\kappa_-(\xi)$ is monotone increasing and such that
\begin{equation*}
	\lim_{\xi\to-\ell^\mp} \kappa_+(\xi)-\kappa_-(\xi)=-\infty
	\qquad\textrm{and}\qquad
	\lim_{\xi\to+\ell^\mp} \kappa_+(\xi)-\kappa_-(\xi)=+\infty.
\end{equation*}
Therefore, there exists unique $\xi_\ast\in(-\ell,\ell)$ such that $(\kappa_+-\kappa_-)(\xi_\ast)=0$
and such a value is such that $\bigl(U^\varepsilon(\cdot;\xi_\ast),v_\ast\bigr)$ is the unique steady
state of the problem.

Finally, denoting by $\omega_1$ the first eigenfunction of the adjoint operator $\mathcal{L}_\xi^\ast$
and setting $\omega_1(x,\xi)=(\phi(x,\xi),\psi(x,\xi))$, equation \eqref{rededo} becomes
\begin{equation*}
	\frac{d\xi}{dt}= -\frac{\bigl(\phi(\xi,\xi),\psi(\xi,\xi)\bigr)\cdot \left(0,\kappa_+(\xi)-\kappa_-(\xi)\right)}
			{\bigl(\phi(\xi,\xi),\psi(\xi,\xi)\bigr)\cdot \left(u_+-u_-,0\right)}
\end{equation*}
that gives
\begin{equation}\label{rededoNS}
	\frac{d\xi}{dt}=-\frac{\psi(\xi,\xi)}{\phi(\xi,\xi)}\frac{\kappa_+(\xi)-\kappa_-(\xi)}{u_+-u_-}.
\end{equation}
The next step is to determine an appropriate (approximate) representation of the ratio
$\psi(\xi,\xi)/\phi(\xi,\xi)$.
This consists in analyzing in details the eigenvalue problem for the adjoint operator of
the linearization at $(U(x,\xi),v_\ast)$.

\subsection*{First eigenfunction of the adjoint operator}

As a first step, let us derive the specific expression of the linearized problem.
Given $W(x,\xi):=(U(x,\xi),v_\ast)$, let us look for solution to \eqref{ns} in the form $W+(u,v)$.
The perturbation $(u,v)$ satisfies
\begin{equation*}
	\begin{aligned}
	&\frac{\partial u}{\partial t}=-\frac{\partial v}{\partial x}-\frac{\partial U}{\partial\xi}\frac{d\xi}{dt},\\
	&\frac{\partial v}{\partial t}=\frac{\partial}{\partial x}\Bigl\{\rho(x)
			+a_1(x)\,u+a_2(x)\,v+b_1(x)\frac{\partial u}{\partial x}+b_2(x)\frac{\partial v}{\partial x}
			+h.o.t.\Bigr\}
	\end{aligned}
\end{equation*}
where
\begin{equation*}
	\begin{aligned}
		\rho(x)&:=-F(U(x),v_\ast)-\varepsilon\,(G(U(x),v_\ast))'\\
		a_1(x)&:=-\partial_u F(U(x),v_\ast)-\varepsilon\,\partial_{uu} G(U(x),v_\ast)U'(x)\\
		a_2(x)&:=-\partial_v F(U(x),v_\ast)-\varepsilon\,\partial_{uv} G(U(x),v_\ast)U'(x)\\
		b_1(x)&:=-\varepsilon\,\partial_u G(U(x),v_\ast)\\
		b_2(x)&:=\varepsilon\nu(U(x))\,U^{-1}(x)\\
	\end{aligned}
\end{equation*}
and $h.o.t.$ collects higher order terms.
Then, the linear operator $\mathcal L_\xi$ is
\begin{equation*}
	\mathcal L_\xi \begin{pmatrix} u \\ v\end{pmatrix}=\frac{d}{dx}\left\{
		\begin{pmatrix} 0 & -1 \\  a_1(x) & a_2(x)\end{pmatrix}\begin{pmatrix} u \\ v\end{pmatrix}
		+\begin{pmatrix} 0 & 0 \\  b_1(x) & b_2(x)\end{pmatrix}\begin{pmatrix} u' \\ v'\end{pmatrix}
		\right\}
\end{equation*}
Setting $w=(u,v)$ and taking the scalar product against the line vector $\omega=(\phi,\psi)$,
we get
\begin{equation*}
	\begin{aligned}
	\omega\cdot \mathcal L_\xi w&=\frac{d}{dx}\left\{
		\begin{pmatrix} \phi & \psi \end{pmatrix}
		\begin{pmatrix} 0 & -1 \\  a_1 & a_2\end{pmatrix}\begin{pmatrix} u \\ v\end{pmatrix}
		+\begin{pmatrix} \phi & \psi \end{pmatrix}
		\begin{pmatrix} 0 & 0 \\  b_1 & b_2\end{pmatrix}\begin{pmatrix} u' \\ v'\end{pmatrix}
		\right\}\\
	&\qquad 
		-\begin{pmatrix} \phi' & \psi' \end{pmatrix}
		\begin{pmatrix} 0 & -1 \\  a_1 & a_2\end{pmatrix}\begin{pmatrix} u \\ v\end{pmatrix}
		-\begin{pmatrix} \phi' & \psi' \end{pmatrix}
		\begin{pmatrix} 0 & 0 \\  b_1 & b_2\end{pmatrix}\begin{pmatrix} u' \\ v'\end{pmatrix}\\
	&=\frac{d}{dx}\mathcal{J}[w,\omega]+(\mathcal L_\xi^\ast \omega)\cdot w
	\end{aligned}
\end{equation*}
where
\begin{equation*}
	\begin{aligned}
	\mathcal{J}[w,\omega]&=\begin{pmatrix} \phi & \psi \end{pmatrix}
		\begin{pmatrix} 0 & -1 \\  a_1 & a_2\end{pmatrix}\begin{pmatrix} u \\ v\end{pmatrix}\\
		&\quad +\begin{pmatrix} \phi & \psi \end{pmatrix}
		\begin{pmatrix} 0 & 0 \\  b_1 & b_2\end{pmatrix}\begin{pmatrix} u' \\ v'\end{pmatrix}
		-\begin{pmatrix} \phi' & \psi' \end{pmatrix}
		\begin{pmatrix} 0 & 0 \\  b_1 & b_2\end{pmatrix} \begin{pmatrix} u \\ v\end{pmatrix}
	\end{aligned}
\end{equation*}
and 
\begin{equation*}
	\mathcal L_\xi^\ast \begin{pmatrix} \phi \\ \psi\end{pmatrix}=
		\begin{pmatrix} 0 & -a_1(x) \\  1 & -a_2(x)\end{pmatrix}\begin{pmatrix} \phi' \\ \psi'\end{pmatrix}
		+ \frac{d}{dx}\left\{
		\begin{pmatrix} 0 & b_1(x) \\  0 & b_2(x)\end{pmatrix}\begin{pmatrix} \phi' \\ \psi' \end{pmatrix}\right\}
\end{equation*} 
The linearization of the boundary conditions \eqref{bdaryuv2} for the operator $\mathcal{L}_\xi$ gives
\begin{equation}\label{boundary}
		\bigl(U v-v_\ast u\bigr)(\pm\ell)=0,\qquad v(-\ell)=0.
\end{equation}
Thus,  there holds
\begin{equation*}
	\begin{aligned}
	\mathcal{J}[w,\omega]\bigr|_{x=-\ell}&=\psi(-\ell)\bigl\{b_1u'+b_2v'\bigr\}\bigr|_{x=-\ell}=0,\\
	\mathcal{J}[w,\omega]\bigr|_{x=+\ell}&=
		-\phi(+\ell)v(+\ell)+\psi(+\ell)\bigl\{a_1u+a_2v+b_1u'+b_2v'\bigr\}\bigr|_{x=+\ell}=0
	\end{aligned}
\end{equation*}
and the requirement $\mathcal{J}\bigr|_{\partial I}=0$ is satisfied if 
\begin{equation}\label{bdaryadj}
		\phi(+\ell)=0,\qquad \psi(\pm\ell)=0,
\end{equation}
that are the boundary conditions for $\mathcal{L}_\xi^\ast$.

Next, we consider the eigenvalue problem for the adjoint operator $\mathcal{L}_\xi^\ast$
with the aim of determining an approximate expression for the ratio $\psi(\xi)/\phi(\xi)$,
where $(\phi,\psi)$ is the eigenvector relative to the first eigenvalue $\lambda$.
The eigenvalue problem reads as
\begin{equation*}
	\left\{\begin{aligned}
		&-a_1(x)\psi'+\bigl(b_1(x)\psi'\bigr)'=\lambda\phi,\\
		&\phi'-a_2(x)\psi'+\bigl(b_2(x)\psi'\bigr)'=\lambda\psi.
	\end{aligned}\right.
\end{equation*}
with boundary condition \eqref{bdaryadj}.
Setting $\theta:=\psi'$ and $w=(\phi,\psi,\theta)$, after some standard algebraic manipulation,
the above system can be rewritten as
\begin{equation}\label{firstorderadj}
	\frac{dw}{dx}=\mathbb{A}\,w	\quad\textrm{where}\quad
	\mathbb{A}:=
	\begin{pmatrix}
		\lambda\,U/v_\ast 				& \lambda & a_{13} \\
		0							&	0	&	1\\
		-\varepsilon^{-1}\lambda/\partial_u G &	0	& \varepsilon^{-1}\partial_u F/\partial_u G	
	\end{pmatrix}
\end{equation}
with $\partial_u F=\partial_u F(U,v_\ast)$, $\partial_u G=\partial_u G(U,v_\ast)$ and
\begin{equation*}
	a_{13}=-\frac{U}{v_\ast}\left(\frac{v_\ast^2}{U^2}+P'(U)\right)-\varepsilon\frac{\nu'(U)}{U}\frac{dU}{dx}
\end{equation*}
In the limiting regime $\varepsilon\to 0^+$, the profile $U$ tends to a step profile joining at $x=\xi$
the values $u_\pm$.
Thus, we consider the matrix-valued $\mathbb{A}$ as piecewise constant and
a representation for the eigenfunction can be obtained by an appropriate matching at $x=\xi$
of solutions to \eqref{firstorderadj} with $U$ considered as a constant.

In order to determine the spectral decomposition of $\mathbb{A}$ and its behavior in the 
regime $\varepsilon\to 0$, we analyze the characteristic roots $\mu$, solutions to
\begin{equation*}
	\varepsilon\det(\mathbb{A}-\mu I)=-\varepsilon\mu^3
		+\left(\frac{\partial_u F}{\partial_u G}+\varepsilon\frac{\lambda U}{v_\ast}\right)\mu^2
		+\frac{2\lambda U}{\nu(U)}\mu-\frac{\lambda^2}{\partial_u G}=0
\end{equation*}
Plugging expansion for $\mu$ in the form $\mu=\mu_1 \varepsilon^{-1}+\mu_0+o(1)$ as
$\varepsilon\to 0$ and skipping the calculations for shortness, we determine asymptotic
formulas for the three roots $\mu_0, \mu_\pm$ 
\begin{equation}\label{asymptoticmu}
	\mu=\mu_{0}=\frac{1}{\varepsilon}\,\frac{\partial_u F}{\partial_u G}+o(\varepsilon^{-1}),
	\qquad
	\mu=\mu_\pm=C_\pm\lambda+o(1).
\end{equation}
with $C_\pm$ are such that
\begin{equation*}
	\partial_u F(U,v_\ast)\,C^2+2\,v_\ast\,U\,C-1=0,
\end{equation*}
that gives, thanks to the specific form of the function $F$,
\begin{equation}\label{defC}
	C_\pm:=\frac{2}{v_\ast U^{-1}\pm\sqrt{P'(U)}}.
\end{equation}
Corresponding expression for the right/left eigenvectors $r/\ell$of $\mathbb{A}$ relative to the eigenvalues
$\mu_{0,\pm}$ can be determined starting from the relations
\begin{equation*}
	\left\{\begin{aligned}
		&\mu\,r_2-r_3=0,\\ 
		&\lambda r_1 + \bigl(\varepsilon\partial_u G\mu-\partial_u F\bigr)r_3=0.
	\end{aligned}\right.
	\qquad
	\left\{\begin{aligned}
		&\left(\frac{\lambda U}{v_\ast}-\mu\right)r_1
			-\frac{1}{\varepsilon}\,\frac{\lambda}{\partial_u G}r_3=0.\\
		&\lambda\,r_1-\mu\,r_2=0,\\
	\end{aligned}\right.
\end{equation*}
where $r=(r_1,r_2,r_3)$  and $\ell=(\ell_1,\ell_2,\ell_3)$.
Taking advantage of the expansions \eqref{asymptoticmu}, we infer for $\mu_0$ the expression
\begin{equation*}
	r_0=(0,0,1)+o(1),\qquad \ell_0=(-\lambda/\partial_u F,0,1)+o(1)
\end{equation*}
and for $\mu_\pm$, using the notation given in \eqref{defC},
\begin{equation*}
	r_\pm=\frac{1}{1+C_\pm^2\partial_u F }(C_\pm\partial_u F,1,C_\pm\lambda)+o(1),
	\qquad
	\ell_\pm=(C_\pm,1,0)+o(1)
\end{equation*}
Thus, disregarding the $o(1)$-terms in the expression for the right and left eigenvectors,
we may introduce the projection matrices
\begin{equation*}
	\begin{aligned}
	\mathbb{P}_0&:=r_0\otimes\ell_0
		=\frac{1}{\partial_u F}\begin{pmatrix}
				0 & 0 & 0 \\ 0 & 0 & 0 \\ -\lambda & 0 & \partial_u F
		\end{pmatrix}
		\\
	\mathbb{P}_\pm&:=r_\pm\otimes\ell_\pm
		=\frac{1}{1+C_\pm^2\partial_u F }\begin{pmatrix}
				C_\pm^2 \partial_u F		& C_\pm \partial_u F & 0 \\
				C_\pm				& 1				& 0 \\
				C_\pm^2\lambda		& C_\pm \lambda	& 0
		\end{pmatrix}
	\end{aligned}
\end{equation*}
Finally, we get an approximated expression for the exponential matrix of $\mathbb{A}$
\begin{equation*}
	e^{\mathbb{A}x}\approx P_0 e^{\mu_0 x}+P_- e^{\mu_- x}+P_+ e^{\mu_+ x}.
\end{equation*}
Considering the solution in the sub-interval $(-\ell,\xi)$ and taking into account
the boundary condition $\psi(-\ell)=0$, we deduce
\begin{equation*}
	w(\xi)=\bigl(\phi(\xi),\psi(\xi),\theta(\xi)\bigr)
		=e^{\mathbb{A}(\xi+\ell)}\bigl(\phi(0),0,\theta(0)\bigr).
\end{equation*}
Hence, we deduce
\begin{equation*}
	\begin{aligned}
	\frac{\phi(\xi)}{\phi(0)} &\approx \partial_u F\left\{\frac{C_-^2}{1+C_-^2\partial_u F }\,e^{\mu_- (\xi+\ell)}
		+\frac{C_+^2}{1+C_+^2\partial_u F } e^{\mu_+ (\xi+\ell)}\right\}\\
	\frac{\psi(\xi)}{\phi(0)} &\approx \frac{C_-}{1+C_-^2\partial_u F } e^{\mu_- (\xi+\ell)}
		+\frac{C_+}{1+C_+^2\partial_u F } e^{\mu_+ (\xi+\ell)}
	\end{aligned}
\end{equation*}
where $\partial_u F$ is calculated at $(u_-,v_\ast)$.
For $\lambda\to 0$, there holds
\begin{equation*}
	\frac{\psi(\xi)}{\phi(0)}
		\approx \frac{C_-}{1+C_-^2\partial_u F }+\frac{C_+}{1+C_+^2\partial_u F }+o(1)
		=\frac{(C_-+C_+)(1+C_-C_+\partial_u F)}{(1+C_+^2\partial_u F)(1+C_+^2\partial_u F)}+o(1)
		=o(1)
\end{equation*}
having used the relation $C_-C_+=-1/\partial_u F$.
Thus, we infer from the expansion for $\mu_\pm$, the relation
\begin{equation*}
	\frac{\psi(\xi)}{\phi(0)}
		\approx \left\{\frac{C_-^2}{1+C_-^2\partial_u F }
			+\frac{C_+^2}{1+C_+^2\partial_u F } \right\}(\xi+\ell)\lambda+o(\lambda)
		=\frac{\phi(\xi)}{\phi(0)}\frac{(\xi+\ell)}{\partial_u F}\,\lambda+o(\lambda)
\end{equation*}
Hence, we end up with the asymptotic expression
\begin{equation*}
	\frac{\psi(\xi)}{\phi(\xi)}\approx \frac{\xi+\ell}{\partial_u F(u_-,v_\ast)}\,\lambda.
\end{equation*}
Plugging into \eqref{rededoNS}, we deduce the new form
\begin{equation}\label{rededoNS2}
	\frac{d\xi}{dt}\approx -\frac{(\xi+\ell)\lambda_1(\xi)}{\partial_u F(u_-,v_\ast)}\frac{\kappa_+(\xi)-\kappa_-(\xi)}{u_+-u_-}.
\end{equation}
where $\lambda_1(\xi)$ denotes the first eigenvalue of the operator $\mathcal{L}_\xi$.

Note that, since $\partial_u F(u_-,v_\ast)<0$ and $\kappa_+-\kappa_-$ is an increasing function,
the equilibrium point $\xi_\ast$ is stable if and only if $\lambda_1(\xi)<0$, as it should be.

\subsection*{Further simplification of the slow motion equation}

As a last step, we derive an approximate expression for the functions $\kappa_\pm$
showing that the motion is indeed exponentially slow. 
Such functions are implicitly defined by the relations
\begin{equation*}
	\int_{u_\ast}^{u_\pm} \frac{\partial_u G(s,v_\ast)}{\kappa_\pm-F(s,v_\ast)}\,ds
	=-\frac{\xi\mp\ell}{\varepsilon}
\end{equation*}
In the limit $\varepsilon\to 0^+$, the righthand side blows up, hence $\kappa_\pm$
is such that
\begin{equation*}
	\lim_{\varepsilon\to 0^+} \kappa_\pm(\xi)=F(u_\pm,v_\ast).
\end{equation*}
To get a preciser description of $\kappa_\pm$ for small $\varepsilon$, 
we approximate functions $F$ and $\partial_u G$ by
\begin{equation*}
	\begin{aligned}
		F(s,v_\ast) & \approx F(u_\pm,v_\ast)+\partial_u F(u_\pm,v_\ast)(s-u_\pm),\\
		\partial_u G(s,v_\ast) & \approx \partial_u G(u_\pm,v_\ast)
	\end{aligned}
\end{equation*}
Setting $h_\pm(\xi):=\kappa_\pm(\xi)-F(u_\pm,v_\ast)$, we obtain
\begin{equation*}
	\int_{u_\ast}^{u_\pm} \frac{\partial_u G(u_\pm,v_\ast)}
		{h_\pm(\xi)-\partial_u F(u_\pm,v_\ast)(s-u_\pm)}\,ds
		=-\frac{\xi\mp\ell}{\varepsilon}
\end{equation*}
Integrating, it follows
\begin{equation*}
	\frac{\partial_u G(u_\pm,v_\ast)}{\partial_u F(u_\pm,v_\ast)}
		\ln\left(1-\frac{\partial_u F(u_\pm,v_\ast)(u_\ast-u_\pm)}{h_\pm(\xi)}
			\right)
	=-\frac{\xi\mp\ell}{\varepsilon}
\end{equation*}
from which we infer
\begin{equation*}
	h_-(\xi)=-\frac{\partial_u F^-(u_\ast-u_-)}
		{\exp\left\{-\frac{\partial_u F^-}{\partial_u G^-}
		\frac{\xi+\ell}{\varepsilon}\right\}-1},\quad
	h_+(\xi)=\frac{\partial_u F^+(u_+-u_\ast)}
		{\exp\left\{\frac{\partial_u F^+}{\partial_u G^+}
		\frac{\ell-\xi}{\varepsilon}\right\}-1}
\end{equation*}
where $\partial_u F^\pm=\partial_u F(u_\pm,v_\ast)$ and
$\partial_u G^\pm=\partial_u G(u_\pm,v_\ast)$.
Finally, since $F(u_-,v_\ast)=F(u_+,v_\ast)$, we obtain
\begin{equation*}
	\kappa_+(\xi)-\kappa_-(\xi)
		\approx \partial_u F^+(u_+-u_\ast)\exp\left\{-\frac{\partial_u F^+}{\partial_u G^+}
			\frac{\ell-\xi}{\varepsilon}\right\}
			+\partial_u F^-(u_\ast-u_-)\exp\left\{\frac{\partial_u F^-}{\partial_u G^-}s
			\frac{\xi+\ell}{\varepsilon}\right\}
\end{equation*}
as $\varepsilon\to 0$.
Inserting into \eqref{rededoNS2}, we obtain the equation
\begin{equation}\label{rededoNS3}
	\frac{d\xi}{dt}=-\biggl\{\frac{\partial_u F^+}{\partial_u F^-}\frac{u_+-u_\ast}{u_+-u_-}
		\exp\left(-\frac{\partial_u F^+}{\partial_u G^+}\frac{\ell-\xi}{\varepsilon}\right)
			+\frac{u_\ast-u_-}{u_+-u_-}
	\exp\left(\frac{\partial_u F^-}{\partial_u G^-}\frac{\xi+\ell}{\varepsilon}\right)\biggr\}
		(\xi+\ell)\lambda_1(\xi)
\end{equation}
which is th ultimate version of the equation for the dynamics along the 
approximate manifolds.

\section{Conclusions}\label{sec:conclusions}

Equation \eqref{rededoNS3} describes the motion's equation along the 
approximate manifold $\{W(\cdot,\xi)\}$ and it should be considered as an equation
for the movement of the transition layer from $(u_-,v_\ast)$ to $(u_+,v_\ast)$ in 
the bounded interval $(-\ell,\ell)$ due to the boundary conditions.
To our knowledge, this is the first attempt of determining a relation that should 
be capable to quantify the boundary effects in the case of isentropic Navier--Stokes
equation for compressible fluids.
Starting from it, a number of possible issues amenable to future investigations arise.
We propose a list of eventual topics, dictated by our personal taste.

1. The determination of the size and, mainly, the sign of the first eigenvalue $\lambda_1$
of the linearized operator $\mathcal L_\xi$ has not yet been explored. 
In particular, its sign is fundamental to determine stability of the single steady state.
Comparing with the case of scalar Burgers equation, it is expected that asymptotic stability
holds and thus that $\lambda_1$ is negative. 

2. The manifold builded to determine the form of the motion's equation is approximated.
The residual is in any case exponentially small and thus the dynamics remain confinated
in a small neighborhood around such manifold.
In analogy with the results proved for the Burgers equation, it would be interesting to
analyze if there exists an exact invariant manifold, lying close to the manifold of the approximate
solutions, and determine a corresponding equation for the dynamics along it.

3. The equation for the location $\xi$ of the transition layer is almost explicit.
The unique term remaining unknown is the first eigenvalue $\lambda_1$ that could
be determined numerically.
We wonder if it would be possible to verify numerically the reliability of the equation,
bypassing the smallness of the right hand side and, thus, the very long timescale of
the shifting phenomenon.

4. As shown at the beginning of Section 2, the approach is formal but flexible and, in 
principle, capable to be applied to different systems fitting into the general class of
hyperbolic--parabolic systems.
Among others, the extension of the approach to the case of non-isentropic
Navier--Stokes would be particularly intruiguing.

5. Following results available in the literature (see Introduction), the case of the half-line
could also be taken into account.
In particular, to quantify by means of a motion's equation the effect of
the presence of a single boundary seems to be a natural variant of the case considered here



\begin{thebibliography}{99}

\bibitem{BHRZ08}
Barker, B.; Humpherys, J.; Rudd, K.; Zumbrun, K.;
{\it Stability of viscous shocks in isentropic gas dynamics},
Comm. Math. Phys. 281 (2008), no. 1, 231--249. 

\bibitem{BJRZ11}
Barker B.; Johnson M.A.; Rodrigues L.M.; Zumbrun K.;
{\it Metastability of solitary roll wave solutions of the St.Venant
equations with viscosity},
Physica D, to appear.

\bibitem{BeckWayn09}
Beck, M.; Wayne, C. E.;
{\it Using global invariant manifolds to understand metastability 
in the Burgers equation with small viscosity},
SIAM J. Appl. Dyn. Syst. 8 (2009) no. 3, 1043--1065. 

\bibitem{BresDesjGera07}
Bresch, D.; Desjardins, B.; G\'erard-Varet, D.;
{\it On compressible Navier-Stokes equations with density dependent
viscosities in bounded domains},
J. Math. Pures Appl. (9) 87 (2007), no. 2, 227--235. 

\bibitem{BresDesjMeti07}
Bresch, D.; Desjardins, B.; M\'etivier, G.;
{\it Recent mathematical results and open problems about shallow water equations},
in ``Analysis and simulation of fluid dynamics,', 15--31, 
Adv. Math. Fluid Mech., Birkh\"auser, Basel, 2007. 

\bibitem{deGrKara98}
de Groen, P. P. N.; Karadzhov, G. E.;
{\it Exponentially slow traveling waves on a finite interval for Burgers' type equation},
Electron. J. Differential Equations 1998, No. 30, 38 pp.

\bibitem{deGrKara01}
de Groen, P. P. N.; Karadzhov, G. E.;
{\it Slow travelling waves on a finite interval for Burgers'-type equations},
Advanced numerical methods for mathematical modelling. 
J. Comput. Appl. Math. 132 (2001) no. 1, 155--189. 

\bibitem{GerbPert01}
Gerbeau J. F.; Perthame B.;
{\it Derivation of Viscous Saint-Venant System for Laminar Shallow Water.
Numerical Validation},
Discrete Contin. Dyn. Syst. Ser. B 1 (2001) no. 1, 89--102.

\bibitem{HubeSerr96}
Hubert, F.; Serre, D.;
{\it Fast-slow dynamics for parabolic perturbations of conservation laws},
Comm. Partial Differential Equations 21 (1996) no. 9-10, 1587--1608. 

\bibitem{HumpLafiZumb10}
Humpherys, J.; Lafitte, O.; Zumbrun, K.;
{\it Stability of isentropic Navier-Stokes shocks in the high-Mach number limit},
Comm. Math. Phys. 293 (2010), no. 1, 1--36. 

\bibitem{KimNi02}
Kim, Y.-J.; Ni, W.-M.;
{\it On the rate of convergence and asymptotic profile of solutions 
to the viscous Burgers equation},
Indiana Univ. Math. J. 51 (2002) no. 3, 727--752. 

\bibitem{KimTzav01}
Kim, Y.J.; Tzavaras, A.E.;
{\it Diffusive N-waves and metastability in the Burgers equation},
SIAM J. Math. Anal. 33 (2001) no. 3, 607--633. 

\bibitem{KreiKrei86}
Kreiss, G., Kreiss, H.-O.;
{\it Convergence to steady state of solutions of Burgers' equation},
Appl. Numer. Math. 2 (1986) no. 3-5, 161--179. 

\bibitem{KreiKreiLore08}
Kreiss, G., Kreiss, H.-O.; Lorenz, J.;
{\it Stability of viscous shocks on finite intervals}, 
Arch. Ration. Mech. Anal. 187 (2008) no. 1, 157--183. 

\bibitem{LafoOMal95a}
Laforgue, J.G. L.; O'Malley, R.E., Jr.;
{\it Shock layer movement for Burgers' equation},
Perturbation methods in physical mathematics (Troy, NY, 1993). 
SIAM J. Appl. Math. 55 (1995) no. 2, 332--347. 

\bibitem{LiLiXin08}
Li H.-L.; Li J., Xin Z.;
{\it Vanishing of vacuum states and blow-up phenomena of the compressible Navier-Stokes equations},
Comm. Math. Phys. 281 (2008), no. 2, 401--444.

\bibitem{LianGuoLi10}
Lian, R.; Guo, Z.; Li, H.-L.;
{\it Dynamical behaviors for 1D compressible Navier-Stokes equations
with density-dependent viscosity},
J. Differential Equations 248 (2010), no. 8, 1926--1954. 

\bibitem{LiuYu97}
Liu, T.-P.; Yu, S.-H.;
{\it Propagation of a stationary shock layer in the presence of a boundary}, 
Arch. Rational Mech. Anal. 139 (1997), no. 1, 57--82. 
	
\bibitem{MascStra13}
Mascia C.; Strani M.;
{\it Metastability for nonlinear parabolic equations with application to scalar viscous
conservation laws}, 
submitted.

\bibitem{MatsNish85}
Matsumura, A.; Nishihara, K.;
{\it On the stability of travelling wave solutions of a one-dimensional model system
for compressible viscous gas},
Japan J. Appl. Math. 2 (1985), no. 1, 17--25. 

\bibitem{MatsWang10}
Matsumura, A.; Wang, Y.;
{\it Asymptotic stability of viscous shock wave for a one-dimensional isentropic model
of viscous gas with density dependent viscosity},
Methods Appl. Anal. 17 (2010), no. 3, 279--290. 

\bibitem{Nish01}
Nishihara, K.;
{\it Boundary effect on a stationary viscous shock wave for scalar viscous conservation laws},
J. Math. Anal. Appl. 255 (2001), no. 2, 535--550. 
	
\bibitem{ReynWard95b}
Reyna, L.G.; Ward, M.J.;
{\it On the exponentially slow motion of a viscous shock},
Comm. Pure Appl. Math. 48 (1995) no. 2, 79--120.

\bibitem{Shih95}
Shih, S.-D.;
{\it A very slowly moving viscous shock of Burgers' equation in the quarter plane},
Appl. Anal. 56 (1995), no. 1-2, 1--18.

\bibitem{ShizKawa85}
Shizuta, Y.; Kawashima, S.;
{\it Systems of equations of hyperbolic-parabolic type with applications
to the discrete Boltzmann equation},
Hokkaido Math. J. 14 (1985), no. 2, 249--275. 

\end{thebibliography}
\end{document}